\documentclass[12pt]{article}
\usepackage{no-ejc}
\usepackage{mathtools}

\usepackage{amsmath,amssymb}

\usepackage{graphicx}
\usepackage{tikz}

\usepackage[colorlinks=true,citecolor=black,linkcolor=black,urlcolor=blue]{hyperref}

\title{\bf Stability for Intersecting Families \\ of Perfect Matchings}

\author{Nathan Lindzey\\
\small Department of Combinatorics and Optimization\\[-0.8ex]
\small University of Waterloo\\[-0.8ex] 
\small Waterloo, ON. Canada\\
\small\tt nlindzey@uwaterloo.ca
}

\dateline{}{}{}


\begin{document}

\maketitle


\begin{abstract}
%
%
A family of perfect matchings of $K_{2n}$ is \emph{intersecting} if any two of its members have an edge in common.  It is known that if $\mathcal{F}$ is family of intersecting perfect matchings of $K_{2n}$, then $|\mathcal{F}|  \leq (2n-3)!!$ and if equality holds, then $\mathcal{F} = \mathcal{F}_{ij}$ where $ \mathcal{F}_{ij}$ is the family of all perfect matchings of $K_{2n}$ that contain some fixed edge $ij$. In this note, we show that the extremal families are stable, namely, that for any $\epsilon \in (0,1/\sqrt{e})$ and $n > n(\epsilon)$, any intersecting family of perfect matchings of size greater than $(1 - 1/\sqrt{e} + \epsilon)(2n-3)!!$ is contained in $\mathcal{F}_{ij}$ for some edge $ij$.  The proof uses the Gelfand pair $(S_{2n},S_2 \wr S_n)$ along with an isoperimetric method of Ellis.

\end{abstract}

\section{Introduction}

Let $\mathcal{M}_{2n}$ be the collection of perfect matchings of the complete graph $K_{2n}$. A family of perfect matchings $\mathcal{F} \subseteq \mathcal{M}_{2n}$ is \emph{intersecting} if $m \cap m' \neq \emptyset$ for any $m,m' \in \mathcal{F}$. 
 It is known that the largest intersecting families of $\mathcal{M}_{2n}$ are the \emph{canonically intersecting families}, which are of the form $ \mathcal{F}_{ij} = \{ m \in \mathcal{M}_{2n} : ij \in m\} \text{ for some } ij \in E(K_{2n})$,
as witnessed by the following Erd\H{o}s-Ko-Rado-type result.
\begin{theorem}\label{thm:ekr}
\emph{\cite{GodsilMeagher,MeagherM05,Lindzey17}} If $\mathcal{F} \subseteq \mathcal{M}_{2n}$ is an intersecting family, then
$$|\mathcal{F}| \leq (2n-3)!!.$$
Moreover, equality holds if and only if $\mathcal{F}$ is a canonically intersecting family.
\end{theorem}
Given such a characterization, a natural next step in extremal combinatorics is to show \emph{stability}, that large families are close in structure to the extremal families.
Our main result is that the extremal families in Theorem~\ref{thm:ekr} are stable for sufficiently large $n$.
\begin{theorem}\label{thm:main}
For any $\epsilon \in (0,1/\sqrt{e})$ and $n > n(\epsilon)$, any intersecting family of $\mathcal{M}_{2n}$ of size greater than $(1 - 1/\sqrt{e} + \epsilon)(2n-3)!!$ is contained in a canonically intersecting family.
\end{theorem}

Our method of proof was originally used by Ellis~\cite{Ellis12} to prove the bipartite version of our main result, originally conjectured by Cameron and Ku~\cite{CameronK03}. In the sequel~\cite{Ellis11}, he showed this method can also be used to show stability results for \emph{t-intersecting families} of perfect matchings of $K_{n,n}$, that is, families such that any two members share $t$ edges.
%

Theorem~\ref{thm:main} also provides an alternative proof of the characterization of the extremal families in Theorem~\ref{thm:ekr} for sufficiently large $n$; however, one can obtain a characterization holding for all $n$ using polyhedral techniques~\cite{Lindzey17, GodsilMeagher}. It was thought that these polyhedral techniques could be extended to the problem of characterizing the extremal $t$-intersecting families of perfect matchings of $K_{n,n}$~\cite[Theorem 27]{EllisFP11}, but this approach has recently been proven incorrect~\cite{Filmus17}. This refutation has sparked renewed interest in Ellis' method, as it currently provides the simplest proof of the following seminal result in Erd\H{o}s-Ko-Rado combinatorics, that the canonically $t$-intersecting families of perfect matchings of $K_{n,n}$ are the extremal $t$-intersecting families for sufficiently large $n$~\cite[pg. 37]{EllisFF15}. 
\begin{theorem}\label{thm:df}
\emph{\cite{EllisFP11,Ellis11}} Let $t \in \mathbb{N}$. If $\mathcal{F}$ is a $t$-intersecting family of perfect matchings of $K_{n,n}$, then for sufficiently large $n$, we have
\[ |\mathcal{F}| \leq (n-t)!.\]
Moreover, equality holds if and only if $\mathcal{F}$ is a canonically $t$-intersecting family, that is, every member of $\mathcal{F}$ contains a fixed set of $t$ disjoint edges of $K_{n,n}$.
\end{theorem}
A well-known conjecture is that a nonbipartite analogue of Theorem~\ref{thm:df} also holds.
\begin{conjecture}
\cite{Lindzey17,GodsilMeagher} 
Let $t \in \mathbb{N}$. If $\mathcal{F}$ is a $t$-intersecting family of perfect matchings of $K_{2n}$, then for sufficiently large $n$, we have
\[ |\mathcal{F}| \leq (2(n-t)-1)!!.\]
Moreover, equality holds if and only if $\mathcal{F}$ is a canonically $t$-intersecting family, that is, every member of $\mathcal{F}$ contains a fixed set of $t$ disjoint edges of $K_{2n}$.
\end{conjecture}
\noindent This conjecture has resisted the usual combinatorial approaches in Erd\H{o}s-Ko-Rado combinatorics, which is not too surprising as there is also no known combinatorial proof of Theorem~\ref{thm:df}. Our main result suggests a possible algebraic route for characterizing the extremal $t$-intersecting families of $\mathcal{M}_{2n}$ for sufficiently large $n$ and resolving this conjecture.
\section{Combinatorial and Algebraic Preliminaries}

Let $\mathcal{M}_{2n}$ be the collection of perfect matchings of $K_{2n}$.  Since $\mathcal{M}_{2n}$ is in one-to-one correspondence with partitions of $[2n] := \{1,2,\cdots,2n\}$ into parts of size two, we may write any perfect matching as a partition 
\[m = m_1~m_2|m_3~m_4|\cdots|m_{2n-1}~m_{2n} \text{ where } m_i \in [2n].\] 
Let $m^* := 1~2|3~4|\cdots |2n$-$1~2n$ be the \emph{identity perfect matching}. The \emph{symmetric group} $S_{2n}$ on $2n$ symbols acts transitively on $\mathcal{M}_{2n}$ under the following action:
\[ \sigma m = \sigma(m_1)~\sigma(m_2)~|~\sigma(m_3)~\sigma(m_4)~|~\cdots~|~\sigma(m_{2n-1})~\sigma(m_{2n}).\]
It is well-known that the \emph{hyperoctahedral group} $H_n := S_2 \wr S_n$ of order $(2n)!! := 2^nn!$ is the stabilizer of $m^*$. Since perfect matchings are in one-to-one correspondence with cosets of the quotient $S_{2n}/H_n$, it follows that 
$$|\mathcal{M}_{2n}| = (2n-1)!! := 1 \times 3 \times 5 \times \cdots \times (2n-3) \times (2n - 1).$$  
Let $(\!( 2n-1)\!)_k := (2n-1) \times (2n-3) \times \cdots \times (2(n-k+1)-1)$ denote the \emph{odd double falling factorial}, which one may compare to the falling factorial $(n)_k := n(n-1)\cdots(n-k+1)$.

For any two perfect matchings $m,m' \in \mathcal{M}_{2n}$, let $\Gamma(m,m') = \Gamma(m',m)$ be the multiset union $m \cup m'$.  It is not hard to see that this graph is composed of disjoint even cycles.
Let $k$ denote the number of connected components of $\Gamma(m,m')$, and let $2\lambda_i$ denote the number of vertices in a component. For any $m,m' \in \mathcal{M}_{2k}$, if we order the components from largest to smallest by number of vertices, we see that $\Gamma(m,m')$ can be identified with an \emph{(integer) partition} $2\lambda := (2\lambda_1, 2\lambda_2,  \cdots , 2\lambda_k) \vdash 2n$.  When referring to the Ferrer's diagram of a partition $\lambda \vdash n$, we call $\lambda$ a \emph{shape}. For any $\lambda \vdash n$, if there are $k$ parts that all have the same size $\lambda_i$, we use $\lambda_i^k$ to denote the multiplicity. Let $d(m,m'): \mathcal{M}_{2n} \times \mathcal{M}_{2n} \mapsto \lambda(n)$ denote the aforementioned bijection, where $\lambda(n)$ is the set of all integer partitions of $n$. Depending on the context, we shall refer to $d(m,m')$ as the \emph{cycle type of $m'$ with respect to m} (or vice versa since $d(m,m') = d(m',m)$).  If one of the arguments is the identity perfect matching, then we say $d(m^*,m)$ is \emph{the cycle type of m}.  Any part of size 1 of a matching's cycle type is called a \emph{fixed point}. Let $\text{fp}(m)$ be the number of fixed points of the cycle type of $m$.

A \emph{derangement} of $\mathcal{M}_{2n}$ is a perfect matching $m \in \mathcal{M}_{2n}$ such that $\text{fp}(m) = 0$. The number of derangements of $M_{2n}$, denoted as $D_{2n}$, can be counted via a recurrence quite similar to the classic one for permutation derangements:
\[D_{2n} = 2 (n - 1)(D_{2(n - 1)} + D_{2(n - 2)}),\]
where $D_0 = 1$ and $D_2 = 0$. Alternatively, via the principle of inclusion-exclusion we have
\begin{align*}
D_{2n} &= \sum_{k=0}^n (-1)^k \binom{n}{k} (2(n-k)-1)!! = (2n-1)!! \sum_{k=0}^n (-1)^k \frac{(n)_k}{k!(\!(2n-1)\!)_k},
\end{align*}
which after taking limits implies that $D_{2n} = (2n-1)!!~(1/\sqrt{e}+o(1))$.  

To give some insight into the conditions of Theorem~\ref{thm:main}, consider the following intersecting family
\[ \mathcal{H}_{1,2} = \{ m \in \mathcal{F}_{1,2} : m \text{ intersects } (1~3)m^* \} \cup \{(1~3)m^*,(1~4)m^*\}.\]
This family is not contained in any canonically intersecting family, and for every member $m \in \mathcal{H}_{1,2} \setminus \{(1~3)m^*,(1~4)m^*\}$, we have that $\{1,4\},\{2,3\},\{2,4\},\{1,3\} \notin m$ as well as $m \cap \{\{5,6\},\{7,8\},\cdots,\{2n-1,2n\}\} \neq \emptyset$.  
The number of perfect matchings $m \in \mathcal{M}_{2n}$ such that $m \cap m^* = \{\{1,2\}\}$ is $D_{2(n-1)}$. The number of perfect matchings such that $m \cap m^* = \{\{1,2\},\{3,4\}\}$ is $D_{2(n-2)}$.
Since $|\mathcal{F}_{1,2}| = (2n-3)!!$, we see that the number of perfect matchings containing $\{1,2\}$ and an edge of $\{\{5,6\},\{7,8\},\cdots,\{2n-1,2n\}\}$ is 
$$|\mathcal{H}_{1,2}| - 2= (2n-3)!! - D_{2(n-1)} - D_{2(n-2)}  = (1-1/\sqrt{e}+o(1))(2n-3)!!.$$  
Note that relabeling the vertices of $K_{2n}$ gives isomorphic families $\mathcal{H}_{i,j}$ for any edge $ij$.  

The \emph{derangement graph} is the graph $\mathcal{D}_n$ such that two perfect matchings $m,m' \in \mathcal{M}_{2n}$ are adjacent in $\mathcal{D}_n$ if $d(m,m')$ has no parts of size 1. An \emph{independent set} of graph $\Gamma$ is a set of vertices $S \subseteq V(\Gamma)$ such that $uv \notin E(\Gamma)$ for all $u,v \in S$. Nonadjacent perfect matchings in the derangement graph are intersecting, thus its independent sets are intersecting families of perfect matchings.


We now recall some basic facts about finite Gelfand pairs, whose proofs can be found in~\cite{CST,MacDonald95}.  A basic understanding of group theory and finite group representation theory is assumed.  In particular, we use many well-known facts from the representation theory of the symmetric group.  The reader is referred to~\cite{MacDonald95,StanleyV201} for a more thorough treatment.

Let $\mathbb{C}[G]$ be the group algebra $G$ over $\mathbb{C}$, and for any subgroup $K \leq G$, define the subalgebra $C(G,K) := \{ f \in \mathbb{C}[G] : f(kxk') = f(x)~\forall x \in G,~\forall k,k' \in K \}$.
\begin{theorem}\label{thm:gelfandPair}
\emph{\cite{MacDonald95}} Let $K \leq G$ be a finite group.  Then the following are equivalent.
\begin{enumerate}
\item $(G,K)$ is a Gelfand Pair;
\item The induced representation $1 \uparrow_K^G \cong \bigoplus_{i=1}^k V_i$ is multiplicity-free;
\item The algebra $C(G,K)$ is commutative.
\end{enumerate}
\end{theorem}
\noindent Let $(G,K)$ be a Gelfand pair and define $\chi_i$ to be the character of $V_i$.  The functions
\[  \phi_i(x) = \frac{1}{|K|} \sum_{k \in K} \overline{\chi_i} (xk) = \frac{1}{|K|} \sum_{k \in K} \chi_i (x^{-1}k) \]
form an orthogonal basis for $C(G,K)$ and are called the \emph{spherical functions}.  It it is helpful to think of the spherical functions as analogues of characters of irreducible representations, as they are constant on double cosets $Kg_iK$.

It is well-known that $(S_{2n}, H_n)$ is a Gelfand pair, which implies the induced representation $1 \uparrow^{S_{2n}}_{H_n}$ admits the following unique decomposition into irreducible representations.
\begin{theorem}
\emph{\cite{Thrall42}}\label{thm:decomp} Let $\lambda = (\lambda_1, \lambda_2,\cdots,\lambda_k) \vdash n$ and $S^{2\lambda}$ be the Specht module of $S_{2n}$ corresponding to the partition $2\lambda := (2\lambda_1, 2\lambda_2,\cdots,2\lambda_k) \vdash 2n$. Then
\[ 1 \uparrow^{S_{2n}}_{H_n} \cong \bigoplus_{\lambda \vdash n} S^{2\lambda}.\]
\end{theorem}
\noindent The eigenspaces of $\mathcal{D}_n$ are precisely the irreducibles $S^{2\lambda}$ stated in the theorem above, and we say that these irreducibles are the \emph{even irreducibles} of $S_{2n}$.
For each $\lambda \vdash n$, let 
$$\Omega_\lambda := \{ m \in \mathcal{M}_{2n} : d(m,m^*) = \lambda \}$$ 
be the \emph{$\lambda$-sphere}, and define the \emph{$\lambda$-double-coset} as $H_n \sigma_{\lambda} H_n = \{\sigma \in S_{2n}: d(m^*,\sigma m^*) = \lambda\}$.
\begin{proposition}\label{lem:sphereSize}
\emph{\cite{MacDonald95}} Let $l(\lambda)$ denote the number of parts of $\lambda \vdash n$, $m_i$ denote the number of parts of $\lambda$ that equal $i$, and set $z_\lambda := \prod_{i \geq 1} i^{m_i} m_i!$.  Then $\Omega_\lambda$ has size
\[ |\Omega_\lambda| = \frac{|H_n|}{2^{l(\lambda)} z_\lambda}.\]
\end{proposition}

\begin{proposition}\label{prop:eigs}
\emph{\cite{Lindzey17}} Let $\Lambda$ be the collection of all integer partitions of $n$ that have no parts of size 1. The eigenvalues $\{\eta_\mu\}_{\mu \vdash n}$ of $\mathcal{D}_n$ can be written as
\[ \eta_\mu = \sum_{\lambda \in \Lambda} |\Omega_\lambda|  \phi^\lambda_\mu\]
where $\{ \phi_\mu \}_{\mu \vdash n}$ are the spherical functions of $(S_n,H_n)$ and $\phi_\mu^\lambda := \phi_\mu(\sigma)$, $\sigma \in H_n \sigma_\lambda H_n$.
\end{proposition}
\noindent For a more detailed discussion of the perfect matching derangement graph, see~\cite{Lindzey17,GodsilMeagher,KuW17}.

\section{The Derangement Graph and the Ratio Bounds}\label{sec:derangement}

The first step in most if not all algebraic proofs of Erd\H{o}s-Ko-Rado-type results is to construct a graph whose independent sets correspond to intersecting families, which in our case is the derangement graph $\mathcal{D}_n$. The following bound of Delsarte and Hoffman has been rather useful for bounding the size of independent sets in such graphs.  
\begin{theorem}[Ratio Bound~\cite{Delsarte73}]\label{thm:pseudoHoffman}
Let $\Gamma$ be a $d$-regular graph with eigenvalues $d = \eta_1 \geq \eta_2 \geq \cdots \geq \eta_{\min}$ and corresponding eigenvectors $v_1, v_2 \cdots , v_{\min}$. If $S\subseteq V$ is an independent set of $\Gamma$, then
\[ |S| \leq |V|\frac{-\eta_{\min}}{d - \eta_{\min}}.\]
If equality holds, then $1_S \in \emph{Span}\left( \{v_1\} \cup \{v_i : \eta_i = \eta_{\min} \} \right)$.
\end{theorem}
\noindent See~\cite{GodsilMeagher} for a comprehensive account of the ratio bound in Erd\H{o}s-Ko-Rado Combinatorics.\\

We now give a short proof that the least eigenvalue of $\mathcal{D}_n$ is $\eta_{(n-1,1)} = -D_{2n}/2(n-1)$ and the magnitudes of its eigenvalues, aside from the least and greatest, are $O((2n-5)!!)$. The latter will be an essential ingredient in our proof of Theorem~\ref{thm:main}.

For any shape $\lambda \vdash n$, we let $S^\lambda$ denote the \emph{irreducible representation of $S_n$} corresponding to $\lambda$ and define $f^\lambda := \dim~S^\lambda$.  We say that an irreducible $S^\lambda$ is \emph{even} if all the parts of $\lambda$ have even size. Let $\rho \downarrow^G_K$ denote the \emph{restriction} of the representation $\rho$ of $G$ to $K$.
\begin{theorem}[The Hook Rule~\cite{Sagan}]
For any shape $\lambda \vdash n$ and cell $c \in \lambda$, let $h(c)$ denote the total number of cells below $c$ in the same column, and to the right of $c$ in the same row including $c$.  Then $ f^\lambda = n!/\prod_{c \in \lambda} h(c)$.
\end{theorem}
\begin{theorem}[The Branching Rule~\cite{Sagan}]
For any irreducible representation $S^\mu$ of $S_n$, we have
$$ S^\mu \downarrow^{S_n}_{S_{n-1}} \cong \bigoplus_{\mu^-} S^{\mu^-}$$ 
where $\mu^-$ ranges over all shapes obtainable from $\mu$ by removing a cell $c$ such that $h(c) = 1$. 
\end{theorem}
\noindent The following result is a well-known and easy to prove consequence of the branching rule.
\begin{corollary}\label{cor:nomults}
For any $\mu \vdash m$ and $2 \leq i < m$ such that $\mu \neq (m)$ or $(1^m)$, the representation $S^\mu \downarrow^{S_m}_{S_{m-i}}$ is reducible.  Moreover, if $S^{2\mu}$ is an even irreducible and $1 \leq i < m$, then the representation $S^{2\mu} \downarrow^{S_{2m}}_{S_{2m-2i}}$ contains at least two even irreducibles unless $2\mu = (2m),(2)^m$.
\end{corollary}
\noindent A technique of James and Kerber~\cite{JamesKerber} allows us to obtain lower bounds on the degrees of even irreducibles of $S_{2n}$ that are not too small in reverse-lexicographical order. For the following proof, it is convenient to abuse notation and let $\lambda \vdash n$ also denote $S^\lambda$.
\begin{lemma}\label{lem:bound}
For $n \geq 8$, the only even irreducibles $\lambda$ of $S_{2n}$ such that $f^\lambda < \binom{2n-4}{4}-\binom{2n-4}{3}$ are $(2n)$ and $(2n-2,2)$. 
\end{lemma}
\begin{proof}
We proceed by induction on $n \geq 8$. Suppose the claim is true for $S_{2(n-1)}$, but not true for $S_{2n}$. Let $\lambda \vdash 2n$ be an even partition such that $f^\lambda < \binom{2n-4}{4}-\binom{2n-4}{3}$. 

If $\lambda \downarrow^{S_{2n}}_{S_{2(n-1)}}$ contains $(2n-2)$ or $(2n-4,2)$ as an irreducible representation, then by the branching rule, the only possibilities for $\lambda$ are $(2n),(2n-2,2),(2n-4,4),$ and $(2n-4,2^2)$, as illustrated below.
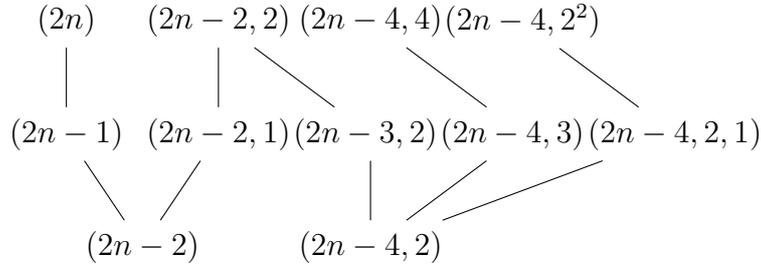
\begin{center}
\begin{tikzpicture}
  \node (max) at (-2,3) {$~~(2n)~~$};
  \node (max2) at (0,3) {$~~(2n-2,2)~~$};
  \node (max3) at (2,3) {$~~(2n-4,4)~~$};
  \node (max4) at (4,3) {$~~(2n-4,2^2)~~$};
  \node (a) at (-2,1.5) {$~~(2n-1)~~$};
  \node (b) at (0,1.5) {$~~(2n-2,1)~~$};
  \node (c) at (2,1.5) {$~~(2n-3,2)~~~$};
  \node (cc) at (4,1.5) {$~~(2n-4,3)~~~~$};
  \node (ccc) at (6,1.5) {$~~(2n-4,2,1)~~$};
  \node (d) at (-1,0) {$~~(2n-2)~~$};
  \node (f) at (2,0) {$~~(2n-4,2)~~$};
  \draw (d) -- (a) -- (max)  (f) -- (cc) -- (max3)
  (f) -- (ccc) -- (max4)
 (f) -- (c)  -- (max2) -- (b)
  (d) -- (b);
\end{tikzpicture}
\end{center}
By the hook formula, we have
$$ f^\lambda < \binom{2n-4}{4}-\binom{2n-4}{3} = f^{(2n-4,4)} < f^{(2n-4,2^2)},$$ which rules out $(2n-4,4)$ and $(2n-4,2^2)$.  We conclude that $(2n-2)$ and $(2n-4,2)$ are not constituents of $\lambda \downarrow^{S_{2n}}_{S_{2(n-1)}}$.  

By the induction hypothesis, all other even irreducibles $\mu < (2n-4,2)$ of $S_{2(n-1)}$ have 
$$f^\mu \geq \binom{2(n-1)-4}{4} - \binom{2(n-1)-4}{3}.$$  
Moreover, for $n\geq8$ we have 
$$2\left(\binom{2(n-1)-4}{4} - \binom{2(n-1)-4}{3}\right) \geq \binom{2n-4}{4}-\binom{2n-4}{3}.$$
Corollary~\ref{cor:nomults} implies that $\lambda = (2n),(2^n)$.  Since $f^{(2^n)} = \frac{1}{n+1}\binom{2n}{n} > \binom{2n-4}{4} - \binom{2n-3}{3}$, we have $\lambda = (2n)$. We conclude that the claim holds for $S_{2n}$, a contradiction.\end{proof}
\noindent The following folklore result gives a crude upperbound on $|\eta_\lambda|$ such that $\lambda \neq (n),(n-1,1)$.
\begin{lemma}[The Trace Bound]\label{lem:trace}
Let $\Gamma$ be a graph on $N$ vertices with eigenvalues $\{\eta_i\}_{i=1}^N$.  Then $\sum_{i=1}^N \eta_i^2 = \emph{Tr}(A(\Gamma)^2) = 2|E(\Gamma)|.$
\end{lemma}
\begin{lemma}\label{lem:bounds}
For all $\lambda  \neq (n), (n-1,1)$, we have $|\eta_{\lambda}| = O(2n-5)!!$.
\end{lemma}
\begin{proof}
By Lemma~\ref{lem:trace} we have $\sum_{\lambda \vdash n} (\sqrt{\dim 2\lambda}~\eta_\lambda)^2 = ((2n-1)!!)^2(1/\sqrt{e} + o(1))$,
thus
\begin{align*}
|\eta_{\lambda}| &\leq \sqrt{\frac{(2n-1)!!^2(1/\sqrt{e} + o(1))}{\dim 2\lambda}} = \frac{(2n-1)!!}{\sqrt{\dim 2\lambda}} \sqrt{1/\sqrt{e} + o(1)} = O((2n-5)!!),
\end{align*}
where the last equality follows from Lemma~\ref{lem:bound}.
\end{proof}
\begin{lemma}\label{lem:zonal}
\emph{\cite[Ch. VII]{MacDonald95}} Let $\phi_{(n-1,1)}^\lambda$ be the zonal spherical function of the Gelfand pair $(S_{2n},H_n)$ that corresponds to $(n-1,1)$ evaluated at $\Omega_\lambda$. Then 
$$\phi_{(n-1,1)}^\lambda = \frac{(2n-1)\emph{fp}(\lambda)-n}{2n(n-1)}.$$
\end{lemma}
\noindent At the expense of using Gelfand pairs, we arrive at a shorter proof of the following.
\begin{theorem}[Godsil and Meagher~\cite{GodsilM15}]
The minimum eigenvalue of the perfect matching derangement graph is $\eta_{(n-1,1)} = -D_{2n}/2(n-1)$.
\end{theorem}
\begin{proof}
By Lemma~\ref{lem:bounds}, only $\eta_{(n)} = D_{2n}$ and $|\eta_{(n-1,1)}|$ are $\omega((2n-5)!!)$.  Derangements have no singleton parts, thus Lemma~\ref{lem:zonal} implies that $\phi_{(n-1,1)}^\lambda = -\frac{1}{2(n-1)}$ for any derangement $\lambda$. By Proposition~\ref{prop:eigs},  we have $\eta_{(n-1,1)} = -D_{2n}/2(n-1)$, as desired.
\end{proof}
\noindent A simple application of the ratio bound proves the first part of Theorem~\ref{thm:ekr}.

We say two families $\mathcal{F}, \mathcal{G} \subseteq \mathcal{M}_{2n}$ are \emph{cross-intersecting} if $m \cap m' \neq \emptyset$ for all $m \in \mathcal{F}$ and $m' \in \mathcal{G}$.
Using the so-called \emph{cross-ratio bound}, we easily obtain Theorem~\ref{cor:cross}, a ``cross-independent" version of the first part of Theorem~\ref{thm:ekr}.
\begin{theorem}[Cross-Ratio Bound~\cite{AlonKKMS02}]
Let $\Gamma$ be a $d$-regular with eigenvalues $d= |\eta_1| \geq |\eta_2| \geq \cdots \geq |\eta_{n}|$ and corresponding eigenvectors $v_1, v_2 \cdots , v_n$. Let $S,T \subset V$ be sets of vertices such that there are no edges between $S$ and $T$. Then 
\[\sqrt{\frac{|S||T|}{|V|^2}} \leq \frac{|\eta_2|}{d + |\eta_2|}.\] 
\end{theorem}
\begin{theorem}\label{cor:cross}
If $\mathcal{F},\mathcal{G} \subseteq \mathcal{M}_{2n}$ are cross-intersecting, then $|\mathcal{F}| \cdot |\mathcal{G}| \leq ((2n-3)!!)^2.$
\end{theorem}
\noindent Let $\mathcal{H}$ be the graph over $\mathcal{M}_{2n}$ such that $m,m'$ are adjacent if and only if $m \cup m'$ is a Hamiltonian cycle of $K_{2n}$.  Similarly, let $\mathcal{H}'$ be the graph over $\mathcal{M}_{2n-1}$ such that $m,m'$ are adjacent if and only if $m \cup m'$ is a Hamiltonian path of $K_{2n-1}$.  Observe that any maximum matching of $K_{2n-1}$ can be extended to a unique perfect matching of $K_{2n}$ by matching the unmatched vertex of $K_{2n-1}$ to the vertex labeled $2n$, and vice versa. This gives a bijection between Hamiltonian paths of $K_{2n-1}$ and Hamiltonian cycles of $K_{2n}$, and shows that $\mathcal{H} \cong \mathcal{H} '$.  This paired with~\cite[Corollary 5.2]{Lindzey17} implies the following. 
\begin{lemma}\label{lem:eig}
The minimum eigenvalue of $\mathcal{H}'$ is $-|H_{n-2}| = -2^{n-2}(n-2)!.$
\end{lemma}
\begin{lemma}\label{lem:oddCross}
If $\mathcal{F}, \mathcal{G} \subseteq \mathcal{M}_{2n-1}$ are cross-intersecting, then $|\mathcal{F}| \cdot |\mathcal{G}| \leq ((2n-3)!!)^2$.
\end{lemma}
\begin{proof}
Note that $\mathcal{H}'$ is a subgraph of the maximum matching derangement graph (two maximum matchings of $K_{2n-1}$ adjacent iff they share no edges). It follows that any pair of cross-intersecting families of maximum matchings of $K_{2n-1}$ are cross-independent sets in $\mathcal{H} '$. Lemma~\ref{lem:eig} together with the cross-ratio bound gives the result.
\end{proof}
\noindent For any intersecting family $\mathcal{F} \subseteq \mathcal{M}_{2n}$, we define the restriction $\mathcal{F} \downarrow_{ij} \subseteq \mathcal{F}$ as the subfamily of members that all contain the edge $ij$, formally, $\mathcal{F} \downarrow_{ij} := \{ m \in \mathcal{F} : ij \in m \}$.
\begin{lemma}\label{lem:cross}
Let $\mathcal{F} \subseteq \mathcal{M}_{2n}$ be an intersecting family. Then for all $i$, $j$ and $k$ with $j \neq k$, we have 
\[|\mathcal{F} \downarrow_{ij} | \cdot |\mathcal{F} \downarrow_{ik}| \leq ((2n - 5)!!)^2.\]
\end{lemma}
\begin{proof}
Without loss of generality, assume $i = 1$, $j = 2,$ and $k=3$.
Note that $ \mathcal{F} \downarrow_{12} \cap~\mathcal{F} \downarrow_{13}~= \emptyset$. Assume both restrictions are nonempty; otherwise, the claim is trivial. Since $\mathcal{F}$ is an intersecting family, any two $m \in \mathcal{F} \downarrow_{12}$ and $m' \in \mathcal{F} \downarrow_{13}$ must share an edge of $E(K_{2n} \setminus \{1,2,3\})$.  In other words, $\mathcal{F} \downarrow_{12}$ and $\mathcal{F} \downarrow_{13}$ are isomorphic to two families $\mathcal{G}$ and $\mathcal{G}'$ of $\mathcal{M}_{2n-3}$ that are cross-intersecting.  The result now follows from Lemma~\ref{lem:oddCross}.
\end{proof}

\section{The Transposition Graph and McDiarmid's Bound}
The \emph{perfect matching transposition graph} is the graph $\mathcal{T}_n$ such that $m,m' \in \mathcal{M}_{2n}$ are adjacent if $d(m,m') = (2,1^{n-2})$.  In other words, two perfect matchings $m,m'$ are adjacent if they differ by a \emph{partner swap}, that is, a transposition $\tau$ such that $m' = \tau m$. This graph will be the combinatorial workhorse of our stability result. The \emph{h-neighborhood} of a set $X \subseteq V$ is the set of vertices $N_h(X) := \{ v \in V : \text{dist} (v,X) \leq h\}$ where $\text{dist} (v,X)$ is the length of a shortest path from $v$ to any vertex of $X$. It is instructive to think of these neighborhoods in the perfect matching transposition graph as balls of radius $h$ in a discrete metric space, as perfect matchings in a ball of small radius around some point in the transposition graph are all structurally quite similar, i.e., they share many edges.

Like the permutation transposition graph, the perfect matching transposition graph admits a nice recursive structure. The following is not too hard to show.
\begin{proposition} \label{prop:trans}
The adjacency matrix of the perfect matching transposition graph of $\mathcal{M}_{2n}$ can be written as the following $(2n-1) \times (2n-1)$ block matrix
\[
A(\mathcal{T}_n) \cong 
\begin{bmatrix}
A(\mathcal{T}_{n-1}) & ~ & ~ & ~ \\
~ & A(\mathcal{T}_{n-1}) & ~ & \emph{\Huge{*}}\\
\emph{\Huge{*}} & ~ & \ddots & ~ \\
~  & ~ & ~ & A(\mathcal{T}_{n-1})\\
\end{bmatrix}
\]
where any off-diagonal block in the $*$ region is a $(2n-3)!! \times (2n-3)!!$ permutation matrix.  Furthermore, $\mathcal{T}_n$ has diameter $n-1$.
\end{proposition}

A \emph{partition sequence} of a graph $\Gamma$ is a sequence $\mathcal{P}_0,\mathcal{P}_1, \cdots , \mathcal{P}_m$ of increasingly refined partitions of $\Gamma(V)$ where $\mathcal{P}_0 = \Gamma(V)$ is the trivial partition, $\mathcal{P}_m$ is the discrete partition into singleton blocks, along with a sequence of numbers $c_0,c_1,\cdots,c_m$ with the following property: for each $i \in \{1,2,\cdots,m\}$, whenever $A,B \in \mathcal{P}_i$, and $A,B \subseteq C \in \mathcal{P}_{i-1}$ for some $C$, then there is a bijection $\varphi : A \rightarrow B$ with $d_{\Gamma}(x,\varphi(x)) \leq c_i$ for all $x \in A$. We say that a partition sequence is \emph{nice} if $m = \text{diameter}(\Gamma)$ and $c_i \leq 1$ for all $i \in \{1,2,\cdots,m\}$.

\begin{theorem}[McDiarmid's Bound~\cite{McDiarmid89}]
Let $\Gamma = (V,E)$ be a graph that admits a partition sequence $\{\mathcal{P}_i\}_{i=0}^m, \{c_i\}_{i=0}^m$, and let $X \subset V$ such that $|X| \geq a|V|$ for some $a \in (0,1)$. Then for any $h \in \mathbb{N}$ such that
$$h > h_0 = \sqrt{\frac{1}{2} \sum_{i=0}^m c_i^2 \ln (1/a) },$$ 
the following holds:
\[ N_h(X) \geq \left(1 - \exp \left(\frac{-2(h-h_0)^2}{\sum_{i=0}^m c_i^2} \right) \right)|V|.\]
\end{theorem}
\noindent By Proposition~\ref{prop:trans}, the perfect matching transposition graph admits a nice partition sequence, and so by McDiarmid's bound, we obtain the following.
\begin{proposition}
Let $X \subset \mathcal{M}_{2n}$ such that $|X| \geq a(2n-1)!!$ for some $a \in (0,1)$. Then for any $h \in \mathbb{N}$ such that 
$$h > h_0 = \sqrt{\frac{n}{2} \ln (1/a) },$$
the following holds: 
\[ N_h(X) \geq \left(1 - \exp \left(\frac{-2(h-h_0)^2}{n} \right) \right)(2n-1)!!.\]
\end{proposition}

\section*{Proof of the Key Lemma}
To prove Theorem~\ref{thm:main}, it suffices to show the following lemma, which we demonstrate below.
\begin{lemma}[Key Lemma]
For any $c \in (0,1)$, there exists a $C > 0$ such that the following holds. If
$\mathcal{F} \subset \mathcal{M}_{2n}$ is an intersecting family with $|\mathcal{F}| \geq c(2n-3)!!$, then there exist an edge $ij$ such that $|\mathcal{F} \setminus  \mathcal{F}\downarrow_{ij} | \leq C(2n-5)!!$.
\end{lemma}

\begin{proof}[Proof of Theorem~\ref{thm:main}]
Let $\mathcal{F}$ be an intersecting family such that $|\mathcal{F}| \geq c(2n-3)!!$ and $c \in (1 - 1/\sqrt{e},1)$. By the key lemma,  implies there exists an edge $ij \in E(K_{2n})$ such that $|\mathcal{F} \setminus \mathcal{F} \downarrow_{ij} | = O((2n - 5)!!).$  
This implies that 
\begin{align}\label{eq:2}
 | \mathcal{F} \downarrow_{ij} | \geq (c - O(1/n))(2n-3)!!.
\end{align}
For sake of contradiction, suppose there exists an $m \in \mathcal{F}$ such that $ij \notin m$.  Since any member of $\mathcal{F} \downarrow_{ij}$ must share an edge with $m$, we have that
\[ |\mathcal{F} \downarrow_{ij}| \leq (2n-3)!! - D_{2(n-1)} - D_{2(n-2)} = (1 - 1/\sqrt{e} - o(1))(2n-3)!!.\]
This contradicts~(\ref{eq:2}) for $n$ sufficiently large depending on $c$, completing the proof.
\end{proof}

A few preliminary results are needed before starting the proof of the key lemma.  First in this list is a generalization of the ratio bound.  
\begin{theorem}[Stability Version of Ratio Bound~\cite{Ellis12}]\label{thm:stableratio}
Let $\Gamma = (V,E)$ be a $d$-regular graph on $N$ vertices with eigenvalues $\eta_{\min}, \cdots, \eta_{\max}=d$ ordered from least to greatest, and corresponding orthonormal eigenvectors $v_{\min}, \cdots, v_{\max}$. Define $\mu := \min \{ \eta_i : \eta_i \neq \eta_{\min} \}$.
Let $X \subseteq V$ be a set of vertices of measure $\alpha := |X|/N$ and let $\ell$ denote the number of edges of the subgraph induced by $X$.
Let $D$ be the Euclidean distance from the characteristic function $f$ of $X$ to the subspace $U = \emph{Span} \left( \{v_{\max}\} \cup \{v_i : \eta_i = \eta_{\min}\} \right)$.  Then
\[ D^2 \leq \alpha \frac{(1-\alpha)|\eta_{\min}| - d\alpha}{|\eta_{\min}| - |\mu|}+ 2\ell.\]
\end{theorem}
Theorem~\ref{thm:stableratio} together with the eigenvalue information on $\mathcal{D}_n$ provides us with upperbounds on how far any intersecting family is from $U$.  Recall that equality is met when we apply the ratio bound to $\mathcal{D}_n$, which implies that $1_{\mathcal{F}_{ij}} \in U \cong S^{2(n)} \oplus S^{2(n-1,1)}$.
We are concerned with how far a ``large" intersecting family $\mathcal{F}$ is from $U$ where ``large" means having size $c(2n-3)!!$ for some $c \in (0,1)$. Recall that the Euclidean distance $D$ from $1_{\mathcal{F}}$ to $U$ can be written as $D = \|P_{U^\perp} 1_{\mathcal{F}} \|_2$ where $P_V$ denotes the projection onto any subspace $V \leq \mathbb{R}[\mathcal{M}_{2n}]$. Since $S^{(2n)}$ is the space of constant functions, the projection of any characteristic function $1_{\mathcal{F}} \in \mathbb{R}[\mathcal{M}_{2n}]$ onto $S^{(2n)}$ is just $(|\mathcal{F}|/(2n-1)!!) 1_{\mathcal{M}_{2n}}$.  More generally, we have the following.


\begin{proposition}\label{prop:proj}
\emph{\cite{Lindzey17, CST}}
Let $E_{\mu} : \mathbb{R}[\mathcal{M}_{2n}] \rightarrow S^{2\mu}$ denote the orthogonal projection onto $S^{2\mu}$ where $\mu \vdash n$.  Then 
\[ [E_{\mu} f](m) = \frac{f^{2\mu}}{(2n-1)!!} \sum_{\lambda \vdash n} \left( \sum_{m' : d(m,m') = \lambda} f(m')  \right) \phi_\mu^\lambda.\]
\end{proposition}

\begin{lemma}\label{lem:proj}
The orthogonal projection $E_{(n-1,1)} : \mathbb{R}[\mathcal{M}_{2n}] \rightarrow S^{2(n-1,1)}$ of the characteristic function $f \in \mathbb{R}[\mathcal{M}_{2n}]$ of the family $\mathcal{F} \subseteq \mathcal{M}_{2n}$ can be written as
\[ [E_{(n-1,1)} f](m) = \frac{(n-1)}{n(2n-3)!!} \left( \sum_{ij \in m} | \mathcal{F} \downarrow_{ij} | \right)  - \frac{|\mathcal{F}|}{2(n-1)}\quad \text{for all $m \in \mathcal{M}_{2n}$}\]
\end{lemma}
\begin{proof}
Applying Proposition~\ref{prop:proj} and Lemma~\ref{lem:zonal} gives us,
\begin{align*}
[E_{(n-1,1)} f](m) &= \frac{f^{2(n-1,1)}}{(2n-1)!!} \sum_{\lambda \vdash n} \left( \sum_{m' : d(m,m') = \lambda} f(m')  \right) \omega_{(n-1,1)}(\lambda)\\
&= \frac{f^{2(n-1,1)}}{(2n-1)!!} \sum_{\lambda \vdash n} \left( \sum_{m' \in \mathcal{F} : d(m,m') = \lambda} \omega_{(n-1,1)}(\lambda)\right) \\
&= \frac{f^{2(n-1,1)}}{(2n-1)!!} \sum_{\lambda \vdash n} \left( \sum_{m' \in \mathcal{F} : d(m,m') = \lambda} \frac{(2n-1)\text{fp}(\lambda)-n}{2n(n-1)} \right) \\
&= \frac{f^{2(n-1,1)}}{(2n-3)!! \cdot 2n(n-1)} \sum_{\lambda \vdash n} \left( \sum_{m' \in \mathcal{F} : d(m,m') = \lambda} \text{fp}(\lambda) \right)  - \frac{n|\mathcal{F}|}{2n(n-1)}\\
&= \frac{1}{(2n-5)!! \cdot 2(n-1)} \left( \sum_{ij \in m} | \mathcal{F} \downarrow_{ij} | \right)  - \frac{|\mathcal{F}|}{2(n-1)}
\end{align*}
where the last equality follows from the hook formula and double-counting.
\end{proof}

We now begin the proof of the key lemma. Due to similarities in the asymptotics of perfect matchings and permutations, some steps follow from~\cite{Ellis12} \emph{mutatis mutandis}. Our notation is consistent with~\cite{Ellis12}.

\begin{proof}[Proof of Key Lemma]
Let $\mathcal{F}$ be an intersecting family such that $|\mathcal{F}| \geq c(2n-3)!!$ and $c \in (0,1)$. Let $f$ be the characteristic function of $\mathcal{F}$, and let $\alpha = |\mathcal{F}|/(2n-1)!!$. Let $D$ be the Euclidean distance from $f$ to $U$. By Theorem~\ref{thm:stableratio}, we have
\begin{align*}
D^2 &\leq \alpha \frac{(1-\alpha)D_{2n}/2(n-1) - D_{2n}\alpha}{D_{2n}/2(n-1) - |\mu|}\\
&= \frac{|\mathcal{F}|}{(2n-1)!!}~~ \frac{1-\alpha - 2(n-1)\alpha}{1 - 2(n-1)|\mu|/D_{2n}}\\
&= \frac{|\mathcal{F}|}{(2n-1)!!}~~ \frac{1 - (2n-1)\alpha}{1 - O(1/n)}\\
&\leq \frac{|\mathcal{F}|}{(2n-1)!!}~~ (1 - (2n-1)\alpha)(1 + O(1/n)),\\
\end{align*}
where the penultimate equality uses the fact that $|\mu| = o((2n-3)!!)$ from Lemma~\ref{lem:bounds}. Now pick $\delta < 1$ so that $|\mathcal{F}| \leq (1-\delta)(2n-3)!!$. We have 
$$\| P_{U^\perp} f \|_2^2 = \| f - P_Uf\|_2^2 = D^2 \leq  \delta(1 + O(1/n)) \frac{|\mathcal{F}|}{(2n-1)!!},$$
which tends to zero as $n \rightarrow \infty$. This already shows that $f$ is ``close" to being a linear combination of canonically intersecting families, but we now seek a combinatorial explanation for this proximity.

By Lemma~\ref{lem:proj}, the projection $P_m := [E_{(n)}f+E_{(n-1,1)}f](m)$ of $f(m)$ onto $U$ is
\begin{align}\label{eq:1}
P_m = \frac{1}{(2n-5)!! \cdot 2(n-1)} \left( \sum_{ij \in m} | \mathcal{F} \downarrow_{ij} | \right)  - \frac{|\mathcal{F}|}{2(n-1)} + \frac{|\mathcal{F}|}{(2n-1)!!},
\end{align}
for any $m \in \mathcal{M}_{2n}$. Note that 
\[  \| f - P_Uf \|_2^2 = \frac{1}{(2n-1)!!} \left( \sum_{m\in \mathcal{F}}(1 - P_m)^2 + \sum_{m \not \in \mathcal{F}} P_m^2 \right) \leq \frac{|\mathcal{F}|}{(2n-1)!!} \delta (1 + O(1/n)),\]
which gives us
\begin{align*}
\sum_{m \in \mathcal{F}}(1 - P_m)^2 + \sum_{m \not \in \mathcal{F}} P_m^2 \leq |\mathcal{F}| \delta (1 + O(1/n)).
\end{align*}
Pick $C > 0$ large enough so that 
$$\sum_{m \in \mathcal{F}}(1 - P_m)^2 + \sum_{m \not \in \mathcal{F}} P_m^2 \leq |\mathcal{F}| \delta (1 + O(1/n)) \leq |\mathcal{F}|(1 - 1/n)\delta(1 + C/n).$$ 
By the non-negativity of each term on the left-hand side of (\ref{eq:1}), at least $|\mathcal{F}|/n$ members of $\mathcal{F}$ satisfy $(1 - P_m)^2 < \delta(1 + C/n)$; therefore, there exists a set
$$\mathcal{F}_1 = \{ m \in \mathcal{F} : (1-P_m)^2 < \delta(1 + C/n)\} $$
such that $|\mathcal{F}_1| \geq |\mathcal{F}|/n$.

Similarly, suppose there are more than $$(2n-1)|\mathcal{F}|(1+O(1/n))/2 \geq (1-\delta)(2n-1)!!(1+O(1/n))/2$$ perfect matchings outside of $\mathcal{F}$ having $P_m^2 \geq 2\delta/(2n-1)$.  Then 
\[ \sum_{m \not \in \mathcal{F}} P_m^2 > \frac{2\delta}{(2n-1)} (1-\delta)(2n-1)!!(1+O(1/n))/2 \geq |\mathcal{F}|\delta(1+O(1/n)) \]
a contradiction; thus there also exists a set
$$\mathcal{F}_0 = \{ m \not \in \mathcal{F} : P_m^2 < 2\delta/(2n-1)\} $$
such that 
$$|\mathcal{F}_0| \geq (2n-1)!! - (1-\delta)(2n-1)!!(1+O(1/n)) /2 - (1-\delta)(2n-3)!!.$$
The projections of the elements of $\mathcal{F}_0$ and $\mathcal{F}_1$ are close to 0 and 1 respectively.
We now show that there exists an $m_1 \in \mathcal{F}_1$ and $m_0 \in \mathcal{F}_0$ that are close together in the transposition graph, which implies that the two share many edges.

To this end, we claim that there is a path $p$ connecting $m_0$ and $m_1$ in the transposition graph $\mathcal{T}_n$ of length at most $2\sqrt{n/2 \log n}$.  To see this, take $a := 1/n^{4}$ and $h := 2h_0$ in McDiarmid's bound.  Since 
$$|\mathcal{F}_1| \geq c(2n-3)!!/n \geq (2n-1)!!/n^4,$$ 
McDiarmid's bound gives us
$$ |N_h(\mathcal{F}_1)| \geq \left(1 - \frac{1}{n^4} \right)(2n-1)!!.$$
Since $|\mathcal{F}_0| > (2n-1)!!/n^4$, we have $|\mathcal{F}_0 \cap N_h(\mathcal{F}_1)| \neq \emptyset$, thus there exists a path $p$ in $\mathcal{T}_n$ of length no more than $2\sqrt{n/2 \log n}$, as desired.

The foregoing shows there exist two perfect matchings $m_1 \in \mathcal{F}$, $m_0 \notin \mathcal{F}$ that are structurally quite similar, differing only in $O(\sqrt{n \log (n)})$ partner swaps, yet 
$$1 - \sqrt{\delta(1 +C/n)} < P_{m_1} \text{ and } P_{m_0} < \sqrt{2\delta/n}.$$ 
Combining inequalities reveals that		
$$P_{m_1} - P_{m_0} > (1-\sqrt{\delta} - O(1/\sqrt{n})).$$
By Equation (\ref{eq:1}), this implies that $m_1$ has many more edges in common with members of $\mathcal{F}$ than $m_0$ does, more formally,
$$ \left( \sum_{ij \in m_1} |\mathcal{F} \downarrow_{ij}| \right)   - \left(\sum_{ij \in m_0}| \mathcal{F} \downarrow_{ij} | \right) \geq (2n-5)!! \cdot 2(n-1)(1-\sqrt{\delta} - O(1/\sqrt{n})).$$
For any $m \in \mathcal{M}_{2n}$, let $m(v)$ denote the partner of $v \in V(K_{2n})$. Let $V(p)$ denote the vertices of $p$. Let $I \subseteq V(K_{2n})$ denote the set of vertices whose partner left them somewhere along the way, less dramatically, 
$$I := \{ v \in V(K_{2n}) : m(v) \neq m'(v) \text{ for some } m,m' \in V(p)  \}.$$ 
Clearly $|I| \leq 4\ell$, where $\ell$ is the length of $p$, and for any $ v \notin I$, we have $m(v) = m'(v)$ for all $m,m' \in V(p)$.  We now have
$$ \left( \sum_{ij \in m_1 : i \in I} |\mathcal{F} \downarrow_{ij}| \right)   - \left(\sum_{ij \in m_0 :  i \in I}| \mathcal{F} \downarrow_{ij} | \right) \geq (2n-5)!! \cdot 2(n-1)(1-\sqrt{\delta} - O(1/\sqrt{n})).$$
This of course implies that
$$ \sum_{ij \in m_1 : i \in I} |\mathcal{F} \downarrow_{ij}|  \geq (2n-5)!! \cdot 2(n-1)(1-\sqrt{\delta} - O(1/\sqrt{n})).$$
Averaging gives us
$$ |\mathcal{F} \downarrow_{ij}|  \geq \frac{(2n-5)!! \cdot 2(n-1)}{4\ell}(1-\sqrt{\delta} - O(1/\sqrt{n}))$$
for some $i \in I$.  Now we have
$$ |\mathcal{F} \downarrow_{ij}|  \geq \frac{(2n-5)!! \cdot 2(n-1)}{4\sqrt{n/2\log(n)}}(1-\sqrt{1-c} - O(1/\sqrt{n})) = \omega((2n-5)!!).$$
Lemma~\ref{lem:cross} implies that $|\mathcal{F} \downarrow_{ik}| = o((2n-5)!!)$ for all $k \neq j$. Summing over all $k \neq j$, we have
$$| \mathcal{F} \setminus \mathcal{F} \downarrow_{ij}| = \sum_{k \neq j}  |\mathcal{F} \downarrow_{ik}| = o((2n-3)!!).$$
This gives us
\[ | \mathcal{F} \downarrow_{ij}| = |\mathcal{F}| - | \mathcal{F} \setminus \mathcal{F} \downarrow_{ij}| = (c-o(1))(2n-3)!!.\]
Since $|\mathcal{F} \downarrow_{ij}| = O((2n-3)!!)$, Lemma~\ref{lem:cross} again implies
$$ |\mathcal{F} \downarrow_{ik}| = O((2n-7)!!)$$
for all $k \neq j$.  Summing over all $k \neq j$ again gives
$$| \mathcal{F} \setminus \mathcal{F} \downarrow_{ij}| = \sum_{k \neq j}  |\mathcal{F} \downarrow_{ik}| = O((2n-5)!!),$$
which completes the proof of the key lemma.
\end{proof}

\subsubsection*{Acknowledgements}
I'd like to thank an anonymous reviewer for pointing out some incorrect calculations in a previous draft, and for several comments that substantially improved the readability.
 \bibliographystyle{plain} 
\bibliography{../master}

\end{document}